\documentclass[12pt]{article}

\usepackage{amssymb}
\usepackage{amsthm}
\usepackage{amsmath}
\usepackage{verbatim}
\usepackage{mathrsfs}
\usepackage[colorlinks,citecolor=red]{hyperref}
\usepackage{young}
\usepackage[english]{babel}
\usepackage{tikz}
\usepackage[noblocks]{authblk}

\usepackage[top=1.35in, bottom=1.5in, left=1.25in, right=1.25in]{geometry}

\newtheorem{Theorem}{Theorem}
\newtheorem{Lemma}[Theorem]{Lemma}
\newtheorem{Example}[Theorem]{Example}
\newtheorem{Remark}[Theorem]{Remark}
\newtheorem{Definition}[Theorem]{Definition}

\newtheorem{Proposition}[Theorem]{Proposition}

\numberwithin{Theorem}{section}

\newcommand{\C}{\mathbb{C}}
\newcommand{\N}{\mathbb{N}}

\newcommand{\F}{\mathbb{F}}

\newcommand{\ds}[1]{\displaystyle{#1}}

\newcommand{\mc}[1]{\mathcal{#1}} 
\newcommand{\mf}[1]{\mathfrak{#1}} 
\newcommand{\mt}[1]{\text{#1}}

\begin{document}

\title{ Bruhat Order on Partial Fixed Point Free Involutions }
\author[1]{Mahir Bilen Can}
\author[2]{Yonah Cherniavsky}
\author[3]{Tim Twelbeck}
\affil[1]{\small{Tulane University, New Orleans, USA; mcan@tulane.edu}}
\affil[2]{\small {Ariel University, Israel; yonahch@ariel.ac.il}}
\affil[3]{\small{Tulane University, New Orleans, USA; ttwelbeck@tulane.edu}}

\maketitle

\begin{abstract} 
The order complex of inclusion poset $PF_n$ of Borel orbit closures 
in skew-symmetric matrices is investigated. It is shown that $PF_n$ is an 
EL-shellable poset, and furthermore, its order complex triangulates a ball. 
The rank-generating function of $PF_n$ is computed and the resulting 
polynomial is contrasted with the Hasse-Weil zeta function of the variety of 
skew-symmetric matrices over finite fields. 
\end{abstract}

\section{Introduction}

This paper is a continuation of our earlier investigations 
\cite{CanCherniavskyTwelbeck}, \cite{CanTwelbeck} 
on the Bruhat order on certain sets of involutions, 
and our notation follows these references closely:
\begin{eqnarray*}
\C &:&\ \text{field of complex numbers},\\
S_n &:&\ \text{symmetric group of $n\times n$ permutation matrices}, \\
R_n &:&\ \text{rook monoid of $n\times n$ partial permutation matrices}, \\
I_n &:&\ \text{involutions in $S_n$}, \\
F_n &:&\ \text{fixed-point-free involutions in $S_n$}, \\
PI_n &:&\ \text{partial involutions in $R_n$}, \\
\mt{Mat}_n &:&\ \text{all $n\times n$ matrices over } \C, \\
\mt{Sym}_n &:&\ \text{all $n\times n$ symmetric matrices over } \C, \\
\mt{GL}_n &:&\ \text{invertible $n\times n$ matrices over } \C,\\
\mt{B}_n &:&\ \text{Borel group of invertible upper triangular matrices from $\mt{GL}_n$.}
\end{eqnarray*}
In addition to the above list of notation, we consider 
$\mt{Skew}_n$, the space of all $n\times n$ skew-symmetric matrices over $\C$,
and $PF_n$, the set of all fixed-point-free partial involutions. The purpose of this 
article is to investigate some combinatorial properties of $PF_n$. In some sense, 
this is the final step of our program for showing that the sets of partial permutations 
$R_n, PI_n$, and $PF_n$ all share the same algebraic combinatorial properties.

Let $X$ be a variety on which a Borel group $B$ acts algebraically. 
Let $W$ denote the set of $B$-orbits in $X$, and define the {\em $B$-ordering} 
$\leq$ on $P$ by
\begin{align}\label{A:B-ordering}
\mc{O}_1 \leq \mc{O}_2 \iff \mc{O}_1 \subseteq \overline{\mc{O}_2}, \qquad  \mc{O}_1, \mc{O}_2\in W.
\end{align}
Study of this basic combinatorial set-up is important for group theory. Indeed, 
suppose $G$ is a linear algebraic group with a Borel subgroup $B$. Then the double 
cosets of $B$ in $G$ are equivalent to the orbits of $B\times B$ acting on $X=G$ via 
$(g,h)\cdot x= gxh^{-1}$. Furthermore, $B\times B$-orbits in $X$ are parametrized
by the `Weyl group' of $G$ (the {\em Bruhat-Chevalley decomposition}). 
We have a well-known special case, when $G= \mt{GL}_n$. 
Then, $\mt{B}_n\times \mt{B}_n$-orbits are parametrized by $S_n$, 
and the induced partial ordering is the {\em Bruhat-Chevalley ordering} on $S_n$.

In~\cite{Renner86}, by generalizing Bruhat-Chevalley decomposition to 
linear algebraic monoids, Renner constructs a rich family of orbit posets. 
In particular, among other things, he shows that the orbits of the Borel group action 
\begin{align}\label{A:Renner}
(g,h)\cdot A = g A h^{-1},\ g,h\in \mt{B}_n,\ A\in \mt{Mat}_n.
\end{align}
are parametrized by $R_n$. 
Basic combinatorial properties of $\mt{B}_n\times \mt{B}_n$-ordering 
on $R_n$ are investigated in~\cite{BC12}.

In~\cite{RS90}, Richardson and Springer investigate the Borel orbits
in the setting of symmetric spaces. In particular, they show that 
the set of involutions $I_n$ of $S_n$ parametrizes the Borel orbits 
in the symmetric space $\mt{SL}_n/\mt{SO}_n$, and furthermore, the corresponding 
$\mt{B}_n$-ordering on $I_n$ agrees with the restriction of the Bruhat-Chevalley ordering 
from $S_n$ (see~\cite{RS94}). Here $\mt{SL}_n$ is the special linear group and 
$\mt{SO}_n$ is its special orthogonal subgroup. Also in~\cite{RS90}, they show that 
$\mt{B}_n$-orbits in $\mt{SL}_{n}/\mt{Sp}_{n}$ are parametrized by $F_n\subset I_n$.

The monoid of matrices $\mt{Mat}_n$ can be viewed as a partial compactification of 
$\mt{GL}_n$, and similarly, the set of all symmetric matrices (respectively, set of all 
skew-symmetric matrices) can be viewed as a partial compactification of $\mt{SL}_n/\mt{SO}_n$
(respectively, of $\mt{SL}_{2n}/\mt{Sp}_{2n}$). 
Similar to the construction of $R_n$, by using suitable modifications of the method of 
Gauss-Jordan elimination, it is shown in~\cite{Szechtman07} for $X=\mt{Sym}_n$, 
and in~\cite{Cherniavsky11} for $X=\mt{Skew}_n$ that the $\mt{B}_n$-orbits of the action 
\begin{align}\label{A:B_n orbits on sym and skew}
g \cdot A = \left(g^{-1}\right)^\top Ag^{-1},\ g\in \mt{B}_n,\ A\in X
\end{align}
are parametrized by $PI_n$ and $PF_n$, respectively. 
Further combinatorial properties of the $\mt{B}_n$-ordering on $PI_n$ and on $PF_n$
are investigated by the second author in the papers~\cite{BC12} (joint with E. Bagno) 
and~\cite{Cherniavsky11}.


There is an interesting relation between $PF_{n}$ and the set of invertible involutions: 
Let $x\in PF_{n}$ be a partial fixed-point-free involution with determinant 0. 
We denote by $\tilde{x}$ the completion of $x$ to an involution in $I_{n}$ 
by adding the missing diagonal entries. For example,
$$
x =
\begin{pmatrix}
0 & 0 & 1 & 0 \\
0 & 0 & 0 & 0 \\
1 & 0 & 0 & 0 \\
0 & 0 & 0 & 0
\end{pmatrix}
\rightsquigarrow
\tilde{x} =
\begin{pmatrix}
0 & 0 & 1 & 0 \\
0 & 1 & 0 & 0 \\
1 & 0 & 0 & 0 \\
0 & 0 & 0 & 1
\end{pmatrix}.
$$
Define $\phi : PF_{n} \rightarrow I_{n}$ by setting
\begin{align}\label{bijection}
\phi(x) =
\begin{cases}
\tilde{x}\ & \ \text{if}\ x\in PF_{n}\setminus F_{n}, \\
x\ & \ \text{otherwise}.
\end{cases}
\end{align}
It is not difficult to check that $\phi$ is a bijection between 
$PF_{n}$ and $I_{n}$ such that $\phi(x) = x$ for all $x\in F_{n}$.
Now that we have two sets in bijection with corresponding $\mt{B}_n$-orderings, 
it is natural to ask for their comparison. This is one of the goals of our paper. 


Recall that the order complex $\Delta(P)$ of a poset $P$ is the abstract 
simplicial complex consisting of all chains in $P$.
Important topological information on a simplicial complex is hidden in the
orderings of its facets (which corresponds to the maximal chains in $P$).
If the facets are ordered in a way that the intersection of a facet with 
all the preceding facets is a simplicial subcomplex of codimension 1, 
then the complex is called {\em shellable}. In this case, it is known that 
the simplicial complex has the homology type of a sphere, or of a ball.
For posets, a purely combinatorial criteria for checking the shellability condition is
found by Bj\"orner in~\cite{Bjorner80}, and it is called the ``lexicographic shellability'' of $P$.

A finite graded poset $P$ with a maximum and a minimum element is called 
{\em EL-shellable},
if there exists a map $f=f_{\varGamma}: C(P) \rightarrow \varGamma$ from the
set of covering relations $C(P)$ of $P$ into a totally ordered set $\varGamma$ satisfying
\begin{enumerate}
\item in every interval $[x,y] \subseteq P$ of length $k>0$ there exists a unique saturated chain
$$\mathfrak{c}:\ x_0=x < x_1 < \cdots < x_{k-1} < x_k=y$$ such that the entries of the sequence
\begin{align}\label{JordanHolder}
f(\mathfrak{c}) = (f(x_0,x_1), f(x_1,x_2), \dots , f(x_{k-1},x_k))
\end{align}
are weakly increasing.

\item The sequence (\ref{JordanHolder}) is lexicographically smallest among all sequences of the form
$$(f(x_0,x_1'), f(x_1',x_2'), \dots , f(x_{k-1}',x_k))$$, where $x_0 < x_1' < \cdots < x_{k-1}' < x_k$.

\end{enumerate}

In literature there are different versions of this notion 
and EL-shellability is known to imply the others (see~\cite{Wachs07}).
A brief history of the shellability questions in Borel orbit posets is as follows:
In \cite{Edelman81}, Edelman proves that BC-order on $S_n$ is EL-shellable.
Shortly after, Proctor in \cite{Proctor82} shows that all classical Weyl groups are
EL-shellable.
Around the same time, in \cite{BW82}, Bj\"orner and Wachs show that Bruchat-Chevalley
ordering on all Coxeter groups, as well as on all sets of minimal-length coset representatives (quotients)
in Coxeter groups are ``dual CL-shellable'' (a weaker alternative to EL-shellability). 
A decade after the introduction of CL-shellability, in~\cite{Dyer93}, M. Dyer shows
that Bruhat-Chevalley ordering on all Coxeter groups and all quotients are EL-shellable.
As an application of EL-shellability, using Dyer's methods, in \cite{Williams07}, L. Williams shows
that the poset of cells of a cell decomposition for totally non-negative part of a flag variety
is EL-shellable.
In the papers \cite{H} and \cite{Hultman05} A. Hultman, although avoids
showing lex. shellability, obtains the same topological consequences 
for the Bruhat-Chevalley ordering on ``twisted involutions'' in Coxeter groups.

There are various directions that the results of \cite{BW82} are extended. 
For semigroups, in \cite{Putcha02}, Putcha shows that ``$J$-classes in 
Renner monoids'' are CL-shellable. In \cite{Can12}, the first author shows 
that for the special Renner monoid $R_n$, not only the $J$-classes are lex. 
shellable, but also the whole rook monoid $R_n$ is EL-shellable.
In \cite{CanTwelbeck}, the first and the third authors show that $PI_n$ 
is EL-shellable. In~\cite{CanCherniavskyTwelbeck}, we show that $F_n$ 
is also EL-shellable, and furthermore, its order complex is a ball of appropriate 
dimension.
In \cite{Incitti04}, Incitti shows that $I_n$ is EL-shellable, and in \cite{Incitti06} he shows
that the $B$-order on involutions in all classical Weyl groups are EL-shellable.

Contributing to the above literature, we show in this paper that 
$PF_n$ is an EL-shellable poset. Moreover, we show that the 
order complex of $PF_n$ triangulates a ball of dimension $n(n-1)/2$. 
On the other hand, it is known that the order complex of $I_n$ triangulates
a sphere of dimension $\lfloor n/4 \rfloor$ (see~\cite{Incitti04}, page 255).

The structure of our paper is as follows. In the next section we introduce 
basic notation for poset theory. In particular, we recollect some known,
basic facts about Bruhat-Chevalley ordering on rooks and partial involutions. 
In Subsection~\ref{S:combinatorial approach}, we compare the 
length functions of $PF_n$ and $PI_n$.

Unfortunately, $PF_n$ is not a connected subposet of $PI_n$, hence 
we are not able to directly utilize our earlier results from~\cite{CanTwelbeck}. 
Therefore, we devote all of Section~\ref{S:coverings} 
for the review of the covering relations of $I_n$, $F_n$, and of $PI_n$ 
in order for describing the covering relations of $PF_n$ next.

In Section~\ref{S:main} we present our proof of EL-shellability of $PF_n$. 
As an application of this result, in Section~\ref{S:order complex}, we 
determine the homotopy type of the order complex of the proper part of $PF_n$,
namely $PF_n$ with its smallest and the largest elements excluded. 

In the final section of our paper, we investigate the length-generating functions of 
certain subposets of $PF_n$. In particular, we relate our length generating function 
computations to the number of rational points of the variety of skew-symmetric matrices 
of fixed rank defined over a finite field.

\vspace{.5cm}

\noindent \textbf{Acknowledgement.}
The first and third authors are partially supported 
by the Louisiana Board of Regents Research and Development Grant.

\section{Preliminaries}

{\em Notation:} Let $m$ be a positive integer. We denote the set $\{1,\dots,m\}$ by $[m]$.
The rank of a matrix $x\in \mt{Mat}_n$ is denoted by $rk(x)$.

\subsection{Poset terminology}
\label{S:posetterminology}

All of our posets are assumed to be finite, graded, and furthermore, they are 
assumed to possess a minimal and a maximal element, 
denoted by $\hat{0}$ and $\hat{1}$, respectively.
We reserve the letter $P$ as the name of a generic such poset and 
denote by $\ell: P \rightarrow \N$ (or, by $\ell_P$, if needed) the length function on $P$.
The set of all covering relations in $P$ is denoted by $C(P)$. 
If $(x,y) \in C(P)$, then we write $y \rightarrow x$ to mean that $y$ covers $x$.

Recall that the M\"obius function of $P$ is defined recursively by the formula
\begin{align*}
\mu ([x,x]) &= 1,\\
\mu ([x,y]) &= - \sum_{x \leq z < y} \mu([x,z])
\end{align*}
for all $x \leq y $ in $P$.
As customary, we denote by $\Delta(P)$ the order complex of $P$. 
It is well known that $\mu(\hat{0},\hat{1})$ is equal to the ``reduced Euler characteristic'' 
$\widetilde{\chi}(\Delta(P))$ of the topological realization of $\Delta(P)$. 
See Proposition 3.8.6 in \cite{StanleyEC1}.

Let $\varGamma$ denote a finite totally ordered poset and let $g$ 
be a $\varGamma$-valued function defined on $C(P)$. 
Then $g$ is called an {\em $R$-labeling} of $P$, if for every interval $[x,y]$ in $P$, 
there exists a unique chain
$x=x_1 \leftarrow x_2 \leftarrow \cdots \leftarrow x_{n-1} \leftarrow x_n = y$ 
such that
\begin{align}\label{R-label}
g(x_1,x_2) \leq g(x_2,x_3) \leq \cdots \leq g(x_{n-1},x_n).
\end{align}
Thus, $P$ is EL-shellable, if it has an $R$-labeling $g:C(P)\rightarrow \varGamma$ 
such that for each interval $[x,y]$ in $P$ the sequence (\ref{R-label}) is lexicographically 
smallest among all sequences of the form
$$
(g(x,x_2'),g(x_2',x_3'), \dots , g(x_{k-1}',y)),
$$
where $x \leftarrow x_2 \leftarrow ' \cdots \leftarrow x_{k-1}' \leftarrow y$.

For $S\subseteq [n]$, by $P_S$ we denote the subset $P_S = \{ x\in P:\ \ell(x) \in S\}$,
and denote by $\mu_S$ the M\"obius function of the poset $\hat{P}_S$ that is obtained 
from $P_S$ by adjoining a smallest and a largest element, if they are missing.
For an $R$-labeling $g:C(P)\rightarrow \varGamma$ of $P$, 
it is well known that the quantity $(-1)^{|S|-1} \mu_S (\hat{0}_{\hat{P}_S},\hat{1}_{\hat{P}_S})$
is equal to the number of maximal chains
$x_0=\hat{0} \leftarrow x_1 \leftarrow \cdots \leftarrow x_n=\hat{1}$ in $P$ 
for which the sequence $(g(x_0,x_1),\cdots,g(x_{n-1},x_n)$ has descent set $S$, 
that is to say, for which $\{ i \in [n]:\ g(x_{i-1},x_i) \geq g(x_{i+1},x_i) \} = S$. 
See Theorem 3.14.2 in \cite{StanleyEC1}.

\subsection{$B$-order on partial involutions}

The notation $F_n$, $I_n$, $PI_n$, $R_n$, $S_n$, $\mt{Skew}_n$, 
and $\mt{Sym}_n$ are as in the introduction.

Recall that $R_n$ parameterizes the $\mt{B}_n \times \mt{B}_n$-orbits in 
$\mt{Mat}_n$. 
For the purposes of this paper, it is more natural for us to look at the inclusion poset
of $\mt{B}_n^\top \times \mt{B}_n$-orbit closures in $R_n$, 
which we denote by $(R_n,\leq_{Rook})$. Here $\mt{B}_n^\top$ is the Borel subgroup 
of all lower triangular matrices from $\mt{GL}_n$.

In \cite{Cherniavsky11}, Cherniavsky shows that the Borel orbits in $\mt{Skew}_{n}$ 
are parametrized by those elements $x\in \mt{Skew}_{n}$
such that
\begin{enumerate}
\item the entries of $x$ are either 0,1 or -1,
\item any non-zero entry of $x$ that is above the main diagonal is a +1,
\item in every row and column of $x$ there exists at most one non-zero entry.
\end{enumerate}
Note that when -1's in $x$ are replaced by +1's, the resulting matrix $\tilde{x}$ 
is a partial involution with no diagonal entry.
In other words, $\tilde{x}$ is a fixed-point-free partial involution. 
It is easy to check that this correspondence is a bijection,
hence $PF_{n}$ parameterizes the Borel orbits in $\mt{Skew}_{n}$.

Containment relations among the closures of Borel orbits in $\mt{Skew}_{n}$ 
define a partial ordering on $PF_{n}$. We denote its dual by $\leq_{Skew}$.
Similarly, on $PI_n$ we have the dual of the partial ordering induced from 
the containment relations among the Borel orbit closures in $\mt{Sym}_n$. 
We denote this dual partial ordering by $\leq_{Sym}$.

\subsection{Combinatorial approach to the posets $R_{n}, PI_{n}, PF_{n}$.}
\label{S:combinatorial approach}

There is a combinatorial method for deciding when two elements 
$x$ and $y$ from $(R_{n},\leq_{Rook})$ (respectively, from $(PI_{n},\leq_{Sym})$, 
or from$(PF_{n},\leq_{Skew})$) are comparable with respect to $\leq_{Rook}$ 
(respectively, with respect to $\leq_{Sym}$, or $\leq_{Skew}$).
We denote by $Rk(x)$ the matrix whose $i,j$-th entry is the rank of the upper left $i\times j$ submatrix of
$x$. Hence, $Rk(x)$ is an $n\times n$ matrix with non-negative integer coordinates.
We call $Rk(x)$, the {\em rank-control matrix} of $x$.

Let $A=(a_{i,j})$ and $B=(b_{i,j})$ be two matrices of the same size with real number entries.
We write $A\leq B$ if $a_{i,j} \leq b_{i,j}$ for all $i$ and $j$.
Then
\begin{align}\label{combinatorial_criterion}
x \leq_{Rook} y \iff Rk(y) \leq Rk(x).
\end{align}
The same criterion holds for the posets $\leq_{Sym}$ and $\leq_{Skew}$.


We recall some fundamental facts about the covering relations of 
$\leq_{Sym}$ and $\leq_{Skew}$.  Our references are \cite{BC12} and \cite{Cherniavsky11}.
Let  $Rk(x)=(r_{i,j})_{i,j=1}^{m}$ denote the rank-control matrix of an $m\times m$ matrix $x$.
As a notation we set $r_{0,i}=0$ for $i=0,\dots, m$ and define
\begin{align}
\rho_\leq (x) &= \# \{ (i,j):\ 1 \leq i \leq j \leq n\ \text{and}\ r_{i,j}=r_{i-1,j-1} \},\\
\rho_< (x) &=  \# \{ (i,j):\ 1 \leq i < j \leq n\ \text{and}\ r_{i,j}=r_{i-1,j-1} \}.
\end{align}
Then the length function $\ell_{{PF}_{n}}$ of the poset $PF_{n}$ is equal to the restriction of
$\rho_{<}$ to $PF_{n}$.
Furthermore, $x$ covers $y$ if and only if $Rk(x)\leqslant Rk(y)$ and 
$\ell_{PF_{n}} (x) - \ell_{PF_{n}}(y) = 1$.


The length function of $PF_n$ differs from the length function of 
$PI_n$ in two ways: The ranks of two matrices $y<x$ in $PF_n$ differ 
by a multiple of 2, and the smallest element in $PI_n$ is the identity matrix, 
which is not in $PF_n$. The minimal element in $PF_n$ is given by the matrix 
with the largest rank-control matrix.
This means that in the case when $n$ is even 
$\ell_{PF_n}(x)=\ell_{PI_n}(x)-\frac{n-rk(x)}{2}-\frac{n}{2}$.
We subtract $\frac{n-rk(x)}{2}$ so that the length function increases 
only by 1 if the rank drops by 2 and we subtract $\frac{n}{2}$ because 
the minimal element has to have length zero.
Similarly, when $n$ is odd we have to subtract $\frac{n-1-rk(x)}{2}$ and $\frac{n+1}{2}$.
Summarizing, we see that for all $n$ the length function $\ell_{PF_n}(x)$ of $PF_n$ is given by
\begin{align}
\ell_{PF_n}(x) &= \ell_{PI_n}(x)-\frac{n-rk(x)}{2}-\frac{n}{2} \notag \\
& =\ell_{PI_n}(x)-\frac{2n-rk(x)}{2} \notag \\ 
&=\rho_<(x)-\frac{2n-rk(x)}{2}.\label{A:a formula for ell}
\end{align}

\begin{Example}
When $n=6$, the smallest element is 
$$
\omega_0=\begin{pmatrix}
0&1&0&0&0&0\\
1&0&0&0&0&0\\
0&0&0&1&0&0\\
0&0&1&0&0&0\\
0&0&0&0&0&1\\
0&0&0&0&1&0
\end{pmatrix},
$$
and when $n=5$, the smallest element is
$
\omega_0=\begin{pmatrix}
0&1&0&0&0\\
1&0&0&0&0\\
0&0&0&1&0\\
0&0&1&0&0\\
0&0&0&0&0
\end{pmatrix}.
$

\end{Example}

\section{An EL-labeling of $PF_n$}\label{S:coverings}

We recall some results on the covering relations of $I_n$, $F_n$, 
and of $PI_n$~\cite{Incitti04,CanCherniavskyTwelbeck, CanTwelbeck}.

\subsection{EL-labeling of $I_n$}

For a permutation $\sigma \in S_n$, a {\em rise} of $\sigma$ is a pair of indices 
$1\leq i_1,i_2\leq n$ such that
$$
i_1 < i_2\  \text{and}\ \sigma(i_1)<\sigma(i_2).
$$
A rise $(i_1,i_2)$ is called {\em free}, if there is no $k \in [n]$ such that
$$
i_1<k<i_2\ \text{and}\ \sigma(i_1)<\sigma(k)<\sigma(i_2).
$$
For $\sigma \in S_n$, define its {\em fixed point set, its exceedance set} and 
its {\em defect set} to be
\begin{align*}
I_f(\sigma) &=Fix(\sigma)=\{i \in [n]:\sigma(i)=i\}, \\
I_e(\sigma) &=Exc(\sigma)=\{i \in [n]:\sigma(i)>i\}, \\
I_d(\sigma) &=Def(\sigma)=\{i \in [n]:\sigma(i)<i\},
\end{align*}
respectively.

Given a rise $(i_1,i_2)$ of $\sigma$, its {\em type} is defined to be the pair $(a,b)$, if
$i_1 \in I_a(\sigma)$ and $i_2 \in I_b(\sigma)$, for some $a,b \in \{f,e,d\}$.
We call a rise of type $(a,b)$ an {\em $ab$-rise}.
On the other hand, two kinds of $ee$-rises have to be distinguished from each other;
an $ee$-rise is called {\em crossing}, if
$i_1<\sigma(i_1)<i_2<\sigma(i_2)$, and it is called {\em non-crossing}, 
if $i_1<i_2<\sigma(i_1)<\sigma(i_2)$.

The rise $(i_1,i_2)$ of an involution $\sigma \in I_n$ is called {\em suitable} 
if it is free and if its type is one of the following:
$(f,f), (f,e), (e,f), (e,e), (e,d).$

A {\em covering transformation}, denoted $ct_{(i_1,i_2)}(\sigma)$, 
of a suitable rise $(i_1,i_2)$ of $\sigma$ is the involution obtained 
from $\sigma$ by moving the 1's from the black dots to the white dots as depicted 
in Figure~\ref{CTofIncitti}.

It is shown in \cite{Incitti04} that if $\tau$ and $\sigma$ are two involutions in $I_n$, then
\begin{align*}
\tau \ \text{covers}\ \sigma\ \text{in}\ \leq_{Sym}\  \iff \ \tau = ct_{(i_1,i_2)}(\sigma),\ 
\text{for some suitable rise}\ (i_1,i_2)\ \text{of}\ \sigma.
\end{align*}
Let $\varGamma$ denote the totally ordered set $[n]\times [n]$ with respect to lexicographic ordering.
In the same paper, Incitti shows that the labeling defined by
\begin{align}\label{Incitti's labeling}
f_\varGamma ((\sigma, ct_{(i_1,i_2)}(\sigma))) := (i_1,i_2) \in \varGamma
\end{align}
is an EL-labeling, hence, $(I_n, \leq_{Sym})$ is an EL-shellable poset.

\begin{figure}[htp]
\begin{center}
\scalebox{1}{
\begin{tikzpicture}[scale=.45]

\begin{scope}[xshift=0,yshift=17.5cm]
\begin{scope}[xshift=-10cm,yshift=0cm]
\node at (0,0) {$ff$-rise:};
\end{scope}
\begin{scope}[xshift=2cm,yshift=0cm]
\node at (0,0) {$\longleftarrow$};
\end{scope}
\begin{scope}[xshift=-2cm,yshift=0cm]
\draw[dotted, thick,black] (-2.5,2.5) -- (2.5,2.5) -- (2.5,-2.5) -- (-2.5,-2.5) -- (-2.5,2.5);
\draw[dotted, thick, black] (-2.5,2.5) -- (2.5,-2.5);
\draw[dotted, thick,black] (-2.5,2.5) -- (2.5,2.5) -- (2.5,-2.5) -- (-2.5,-2.5) -- (-2.5,2.5);
\draw[dotted, thick, black] (-2.5,2.5) -- (2.5,-2.5);
\filldraw[dotted,fill=gray!45,opacity=.95]  (-.5,.5) -- (.5,.5) -- (.5,-.5) -- (-.5,-.5) -- (-.5,.5) ;
\node at (.5,.5) {0};
\node at (-.5,-.5) {0};
\node at (-.5,.5) {1};
\node at (.5,-.5) {1};
\node at (-2.8,.5) {$i$};
\node at (-2.8,-.5) {$j$};
\node at (-.6,2.85) {$i$};
\node at (.65,2.85) {$j$};
\end{scope}
\begin{scope}[xshift=6.5cm,yshift=0cm]
\draw[dotted, thick,black] (-2.5,2.5) -- (2.5,2.5) -- (2.5,-2.5) -- (-2.5,-2.5) -- (-2.5,2.5);
\draw[dotted, thick, black] (-2.5,2.5) -- (2.5,-2.5);
\filldraw[dotted,fill=gray!45,opacity=.95]  (-.5,.5) -- (.5,.5) -- (.5,-.5) -- (-.5,-.5) -- (-.5,.5) ;
\node at (.5,.5) {1};
\node at (-.5,-.5) {1};
\node at (-.5,.5) {0};
\node at (.5,-.5) {0};
\node at (-2.8,.5) {$i$};
\node at (-2.8,-.5) {$j$};
\node at (-.6,2.85) {$i$};
\node at (.65,2.85) {$j$};
\end{scope}
\end{scope}

\begin{scope}[xshift=0,yshift=10.5cm]
\begin{scope}[xshift=-10cm,yshift=0cm]
\node at (0,0) {$fe$-rise:};
\end{scope}
\begin{scope}[xshift=2cm,yshift=0cm]
\node at (0,0) {$\longleftarrow$};
\end{scope}
\begin{scope}[xshift=-2cm,yshift=0cm]
\draw[dotted, thick,black] (-2.5,2.5) -- (2.5,2.5) -- (2.5,-2.5) -- (-2.5,-2.5) -- (-2.5,2.5);
\draw[dotted, thick, black] (-2.5,2.5) -- (2.5,-2.5);
\filldraw[dotted,fill=gray!45,opacity=.95] (-1.5,1.5) -- (1,1.5) -- (1,.5) -- (-.5,.5) -- (-.5,-1) -- (-1.5,-1);
\node at (-1.5,1.5) {1};
\node at (1,.5) {1};
\node at (-.5,-1) {1};
\node at (-.5,.5) {0};
\node at (1,1.5) {0};
\node at (-1.5,-1) {0};
\node at (-2.75,1.5) {$i$};
\node at (-2.75,.5) {$j$};
\node at (-3,-.9) {\scriptsize{$\sigma(j)$}};
\node at (-1.5,2.85) {$i$};
\node at (-.6,2.85) {$j$};
\node at (.9,2.85) {\scriptsize{$\sigma(j)$}};
\end{scope}
\begin{scope}[xshift=6.5cm,yshift=0cm]
\draw[dotted, thick,black] (-2.5,2.5) -- (2.5,2.5) -- (2.5,-2.5) -- (-2.5,-2.5) -- (-2.5,2.5);
\draw[dotted, thick, black] (-2.5,2.5) -- (2.5,-2.5);
\filldraw[dotted,fill=gray!45,opacity=.95] (-1.5,1.5) -- (1,1.5) -- (1,.5) -- (-.5,.5) -- (-.5,-1) -- (-1.5,-1);
\node at (-1.5,1.5) {0};
\node at (1,.5) {0};
\node at (-.5,-1) {0};
\node at (-.5,.5) {1};
\node at (1,1.5) {1};
\node at (-1.5,-1) {1};
\node at (-2.75,1.5) {$i$};
\node at (-2.75,.5) {$j$};
\node at (-3,-.9) {\scriptsize{$\sigma(j)$}};
\node at (-1.5,2.85) {$i$};
\node at (-.6,2.85) {$j$};
\node at (.9,2.85) {\scriptsize{$\sigma(j)$}};
\end{scope}
\end{scope}

\begin{scope}[xshift=0,yshift=3.5cm]
\begin{scope}[xshift=-10cm,yshift=0cm]
\node at (0,0) {$ef$-rise:};
\end{scope}
\begin{scope}[xshift=2cm,yshift=0cm]
\node at (0,0) {$\longleftarrow$};
\end{scope}
\begin{scope}[xshift=-2cm,yshift=0cm]
\draw[dotted, thick,black] (-2.5,2.5) -- (2.5,2.5) -- (2.5,-2.5) -- (-2.5,-2.5) -- (-2.5,2.5);
\draw[dotted, thick, black] (-2.5,2.5) -- (2.5,-2.5);
\filldraw[dotted,fill=gray!45,opacity=.95] (0,0) -- (0,1.5) -- (1,1.5)-- (1,-1)  -- (-1.5,-1) 
-- (-1.5, 0) -- (0,0);
\node at (0,1.5) {1};
\node at (-1.5,0) {1};
\node at (1,-1) {1};
\node at (0,0) {0};
\node at (1,1.5) {0};
\node at (-1.5,-1) {0};
\node at (-2.8,1.5) {$i$};
\node at (-2.9,0) {\scriptsize{$\sigma(i)$}};
\node at (-2.8,-1) {$j$};
\node at (-1.5, 2.85) {$i$};
\node at (-.25, 2.85) {\scriptsize{$\sigma(i)$}};
\node at (1, 2.85) {$j$};
\end{scope}
\begin{scope}[xshift=6.5cm,yshift=0cm]
\draw[dotted, thick,black] (-2.5,2.5) -- (2.5,2.5) -- (2.5,-2.5) -- (-2.5,-2.5) -- (-2.5,2.5);
\draw[dotted, thick, black] (-2.5,2.5) -- (2.5,-2.5);
\filldraw[dotted,fill=gray!45,opacity=.95] (0,0) -- (0,1.5) -- (1,1.5)-- (1,-1)  -- (-1.5,-1) 
-- (-1.5, 0) -- (0,0);
\node at (0,1.5) {0};
\node at (-1.5,0) {0};
\node at (1,-1) {0};
\node at (0,0) {1};
\node at (1,1.5) {1};
\node at (-1.5,-1) {1};
\node at (-2.8,1.5) {$i$};
\node at (-2.9,0) {\scriptsize{$\sigma(i)$}};
\node at (-2.8,-1) {$j$};
\node at (-1.5, 2.85) {$i$};
\node at (-.25, 2.85) {\scriptsize{$\sigma(i)$}};
\node at (1, 2.85) {$j$};
\end{scope}
\end{scope}

\begin{scope}[xshift=0,yshift=-3.5cm]
\begin{scope}[xshift=-12.5cm,yshift=0cm]
\node at (0,0) {Non-crossing $ee$-rise:};
\end{scope}
\begin{scope}[xshift=2cm,yshift=0cm]
\node at (0,0) {$\longleftarrow$};
\end{scope}
\begin{scope}[xshift=-2cm,yshift=0cm]
\draw[dotted, thick,black] (-2.5,2.5) -- (2.5,2.5) -- (2.5,-2.5) -- (-2.5,-2.5) -- (-2.5,2.5);
\draw[dotted, thick, black] (-2.5,2.5) -- (2.5,-2.5);
\filldraw[dotted,fill=gray!45,opacity=.95] (.5,.5) -- (1.5,.5) -- (1.5,1.5) -- (.5,1.5) -- (.5,.5) ;
\filldraw[dotted,fill=gray!45,opacity=.95] (-.5,-.5) -- (-1.5,-.5) -- (-1.5,-1.5) -- (-.5,-1.5) -- (-.5,-.5) ;
\node at (.5,.5) {0};
\node at (1.5,.5) {1};
\node at (1.5,1.5) {0};
\node at (.5,1.5) {1};
\node at (-.5,-.5) {0};
\node at (-1.5,-.5) {1};
\node at (-1.5,-1.5) {0};
\node at (-.5,-1.5) {1};
\node at (-2.8,1.5) {$i$};
\node at (-2.8,.5) {$j$};
\node at (-3,-.5) {\scriptsize{$\sigma(i)$}};
\node at (-3,-1.5) {\scriptsize{$\sigma(j)$}};
\node at (-1.5,2.85) {$i$};
\node at (-.5,2.85) {$j$};
\node at (.5,2.85) {\scriptsize{$\sigma(i)$}};
\node at (1.75,2.85) {\scriptsize{$\sigma(j)$}};
\end{scope}
\begin{scope}[xshift=6.5cm,yshift=0cm]
\draw[dotted, thick,black] (-2.5,2.5) -- (2.5,2.5) -- (2.5,-2.5) -- (-2.5,-2.5) -- (-2.5,2.5);
\draw[dotted, thick, black] (-2.5,2.5) -- (2.5,-2.5);
\filldraw[dotted,fill=gray!45,opacity=.95] (.5,.5) -- (1.5,.5) -- (1.5,1.5) -- (.5,1.5) -- (.5,.5) ;
\filldraw[dotted,fill=gray!45,opacity=.95] (-.5,-.5) -- (-1.5,-.5) -- (-1.5,-1.5) -- (-.5,-1.5) -- (-.5,-.5) ;
\node at (.5,.5) {1};
\node at (1.5,.5) {0};
\node at (1.5,1.5) {1};
\node at (.5,1.5) {0};
\node at (-.5,-.5) {1};
\node at (-1.5,-.5) {0};
\node at (-1.5,-1.5) {1};
\node at (-.5,-1.5) {0};
\node at (-2.8,1.5) {$i$};
\node at (-2.8,.5) {$j$};
\node at (-3,-.5) {\scriptsize{$\sigma(i)$}};
\node at (-3,-1.5) {\scriptsize{$\sigma(j)$}};
\node at (-1.5,2.85) {$i$};
\node at (-.5,2.85) {$j$};
\node at (.5,2.85) {\scriptsize{$\sigma(i)$}};
\node at (1.75,2.85) {\scriptsize{$\sigma(j)$}};
\end{scope}
\end{scope}

\begin{scope}[xshift=0,yshift=-10.5cm]
\begin{scope}[xshift=-12cm,yshift=0cm]
\node at (0,0) {Crossing $ee$-rise:};
\end{scope}
\begin{scope}[xshift=2cm,yshift=0cm]
\node at (0,0) {$\longleftarrow$};
\end{scope}
\begin{scope}[xshift=-2cm,yshift=0cm]
\draw[dotted, thick,black] (-2.5,2.5) -- (2.5,2.5) -- (2.5,-2.5) -- (-2.5,-2.5) -- (-2.5,2.5);
\draw[dotted, thick, black] (-2.5,2.5) -- (2.5,-2.5);
\filldraw[dotted,fill=gray!45,opacity=.95] (-.5,.5) -- (-.5,1.5) -- (1.5,1.5)-- (1.5, -.5)  
-- (.5,-.5) -- (.5,-1.5) -- (-1.5,-1.5)-- (-1.5,-.5) -- (-1.5, .5) -- (-.5,.5);
\node at (-.5,1.5) {0};
\node at (1.5,-.5) {0};
\node at (.5,-1.5) {0};
\node at (-1.5,.5) {0};
\node at (1.5,1.5) {1};
\node at (-1.5,-1.5) {1};
\node at (-.5,.5) {1};
\node at (.5,-.5) {1};
\node at (-2.8,1.5) {$i$};
\node at (-2.9,.5) {\scriptsize{$\sigma(i)$}};
\node at (-2.8,-.5) {$j$};
\node at (-3,-1.5) {\scriptsize{$\sigma(j)$}};
\node at (-1.5,2.85) {$i$};
\node at (-.6,2.85) {\scriptsize{$\sigma(i)$}};
\node at (.65,2.85) {$j$};
\node at (1.75,2.85) {\scriptsize{$\sigma(j)$}};
\end{scope}
\begin{scope}[xshift=6.5cm,yshift=0cm]
\draw[dotted, thick,black] (-2.5,2.5) -- (2.5,2.5) -- (2.5,-2.5) -- (-2.5,-2.5) -- (-2.5,2.5);
\draw[dotted, thick, black] (-2.5,2.5) -- (2.5,-2.5);
\filldraw[dotted,fill=gray!45,opacity=.95] (-.5,.5) -- (-.5,1.5) -- 
(1.5,1.5)-- (1.5, -.5)  -- (.5,-.5) -- (.5,-1.5) -- (-1.5,-1.5)-- (-1.5,-.5) -- (-1.5, .5) -- (-.5,.5);
\node at (-.5,1.5) {1};
\node at (1.5,-.5) {1};
\node at (.5,-1.5) {1};
\node at (-1.5,.5) {1};
\node at (1.5,1.5) {0};
\node at (-1.5,-1.5) {0};
\node at (-.5,.5) {0};
\node at (.5,-.5) {0};
\node at (-2.8,1.5) {$i$};
\node at (-2.9,.5) {\scriptsize{$\sigma(i)$}};
\node at (-2.8,-.5) {$j$};
\node at (-3,-1.5) {\scriptsize{$\sigma(j)$}};
\node at (-1.5,2.85) {$i$};
\node at (-.6,2.85) {\scriptsize{$\sigma(i)$}};
\node at (.65,2.85) {$j$};
\node at (1.75,2.85) {\scriptsize{$\sigma(j)$}};
\end{scope}
\end{scope}

\begin{scope}[xshift=0,yshift=-17.5cm]
\begin{scope}[xshift=-10cm,yshift=0cm]
\node at (0,0) {$ed$-rise:};
\end{scope}
\begin{scope}[xshift=2cm,yshift=0cm]
\node at (0,0) {$\longleftarrow$};
\end{scope}
\begin{scope}[xshift=-2cm,yshift=0cm]
\draw[dotted, thick,black] (-2.5,2.5) -- (2.5,2.5) -- (2.5,-2.5) -- (-2.5,-2.5) -- (-2.5,2.5);
\draw[dotted, thick, black] (-2.5,2.5) -- (2.5,-2.5);
\filldraw[dotted,fill=gray!45,opacity=.95] (-.5,.5) -- (-.5,1.5) -- (.5,1.5) -- (.5,.5) -- (1.5,.5) -- (1.5, -.5)  -- (.5,-.5) 
-- (.5,-1.5) -- (-.5,-1.5) -- (-.5, -.5) -- (-1.5,-.5) -- (-1.5, .5) -- (-.5,.5);
\node at (-.5,1.5) {1};
\node at (.5,1.5) {0};
\node at (1.5,.5) {0};
\node at (1.5,-.5) {1};
\node at (.5,-1.5) {1};
\node at (-.5,-1.5) {0};
\node at (-1.5,-.5) {0};
\node at (-1.5,.5) {1};
\node at (-2.8,1.5) {$i$};
\node at (-2.9,.5) {\scriptsize{$\sigma(i)$}};
\node at (-3,-.5) {\scriptsize{$\sigma(j)$}};
\node at (-2.8,-1.5) {$j$};
\node at (-1.5,2.85) {$i$};
\node at (-.6,2.85) {\scriptsize{$\sigma(i)$}};
\node at (.65,2.85) {\scriptsize{$\sigma(j)$}};
\node at (1.75,2.85) {$j$};
\end{scope}
\begin{scope}[xshift=6.5cm,yshift=0cm]
\draw[dotted, thick,black] (-2.5,2.5) -- (2.5,2.5) -- (2.5,-2.5) -- (-2.5,-2.5) -- (-2.5,2.5);
\draw[dotted, thick, black] (-2.5,2.5) -- (2.5,-2.5);
\filldraw[dotted,fill=gray!45,opacity=.95] (-.5,.5) -- (-.5,1.5) -- (.5,1.5) -- (.5,.5) -- 
(1.5,.5) -- (1.5, -.5)  -- (.5,-.5) 
-- (.5,-1.5) -- (-.5,-1.5) -- (-.5, -.5) -- (-1.5,-.5) -- (-1.5, .5) -- (-.5,.5);
\node at (-.5,1.5) {0};
\node at (.5,1.5) {1};
\node at (1.5,.5) {1};
\node at (1.5,-.5) {0};
\node at (.5,-1.5) {0};
\node at (-.5,-1.5) {1};
\node at (-1.5,-.5) {1};
\node at (-1.5,.5) {0};
\node at (-2.8,1.5) {$i$};
\node at (-2.9,.5) {\scriptsize{$\sigma(i)$}};
\node at (-3,-.5) {\scriptsize{$\sigma(j)$}};
\node at (-2.8,-1.5) {$j$};
\node at (-1.5,2.85) {$i$};
\node at (-.6,2.85) {\scriptsize{$\sigma(i)$}};
\node at (.65,2.85) {\scriptsize{$\sigma(j)$}};
\node at (1.75,2.85) {$j$};
\end{scope}
\end{scope}

\end{tikzpicture}
}
\end{center}
\caption{Covering transformations $\sigma \leftarrow \tau=ct_{(i,j)}(\sigma)$ of $I_n$.}
\label{CTofIncitti}
\end{figure}
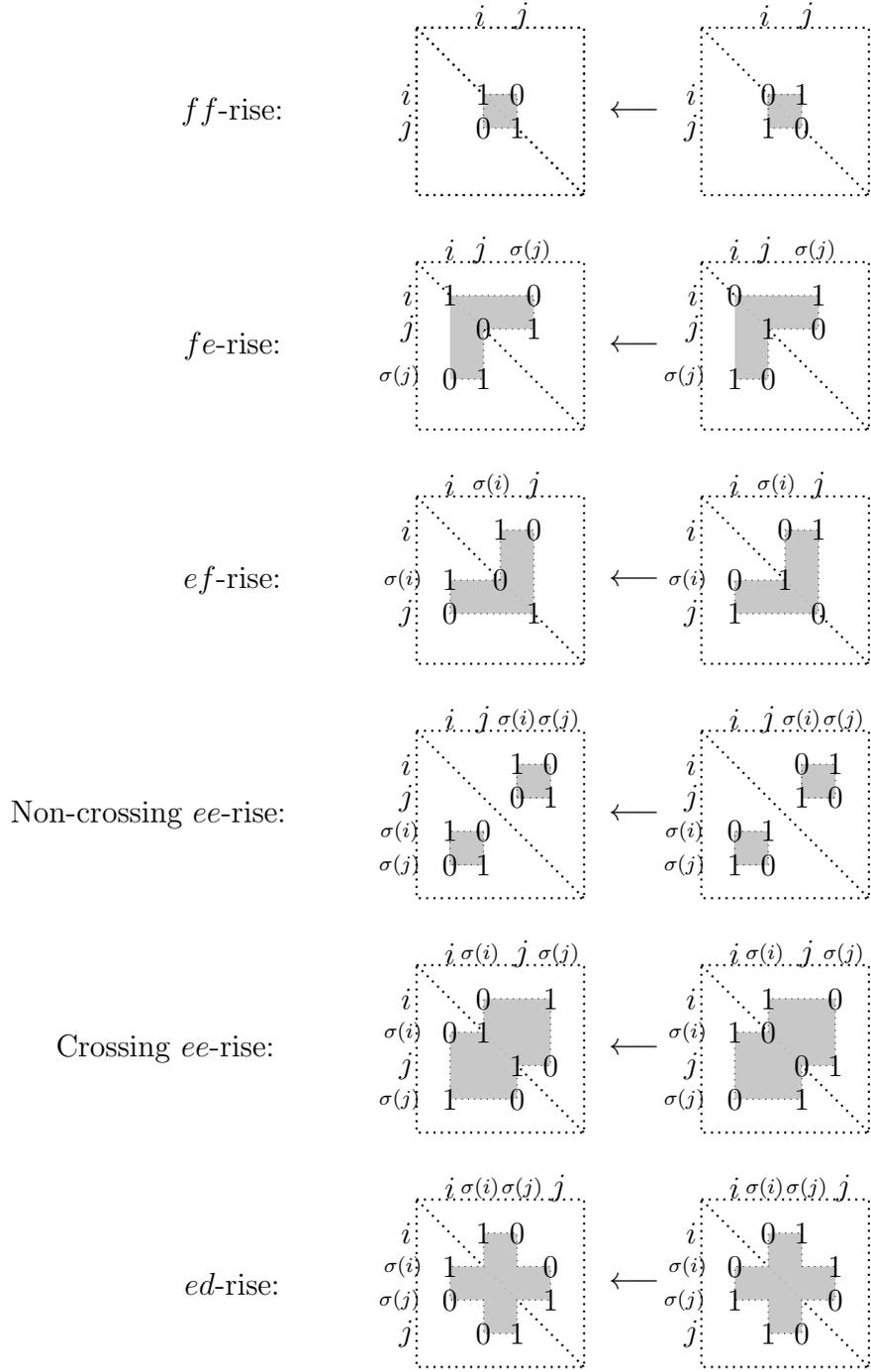

\subsection{EL-labeling of $F_{2n}$}

Recall that $F_{2n}$ is a connected graded subposet of $I_{2n}$. 
Therefore, its covering relations are among the covering relations of $I_{2n}$.
On the other hand, within $F_{2n}$ we use two types of covering transformations, only: 
a non-crossing $ee$-rise and an $ed$-rise.
These moves correspond to the items numbered 4 and 6 in Table 1 of \cite{Incitti04}.
It is shown in~\cite{CanCherniavskyTwelbeck} that these covering labels is an EL-labeling
for $F_{2n}$.

\subsection{EL-labeling of $PI_n$}

When two partial involutions $x$ and $y$ have the same zero rows 
and zero columns, the covering relation $x\rightarrow y$ is not different than 
the invertible case. 
\begin{Example} \label{E:1}
$$
y=\begin{pmatrix}
1&0&0\\
0&0&0\\
0&0&1
\end{pmatrix}
\
\text{is covered by}
\
x=\begin{pmatrix}
0&0&1\\
0&0&0\\
1&0&0
\end{pmatrix}.
$$
Note that $x\rightarrow y$ if and only if the invertible involution $\tilde{x}$, that is 
obtained from $x$ by removing the rows and columns of $x$ with no
non-zero entries, covers the invertible involution $\tilde{y}$ that is obtained 
from $y$ by removing its rows and columns with zeros only. 
\end{Example} 

Moving down a non-zero entry along the diagonal gives a covering relation:
\begin{Example}\label{E:2}
$$
y=\begin{pmatrix}
0&0&1&0\\
0&1&0&0\\
1&0&0&0\\
0&0&0&0\\
\end{pmatrix}
\
\text{is covered by}
\
x=\begin{pmatrix}
0&0&1&0\\
0&0&0&0\\
1&0&0&0\\
0&0&0&1\\
\end{pmatrix}.
$$
Similarly, 
$$
y=\begin{pmatrix}
0&0&0\\
0&1&0\\
0&0&1\\
\end{pmatrix}
\
\text{is covered by}
\
x=\begin{pmatrix}
0&0&0\\
0&1&0\\
0&0&0\\
\end{pmatrix}.
$$
\end{Example}

Another type of covering relation is obtained by the moving of 
off-diagonal pairs $(i,j)$ and $(j,i)$, where $i>j$ to down/right, or 
to right/down available positions.
\begin{Example}\label{E:3}
There are two cases:
\begin{enumerate}
\item
$\ds{
y=\begin{pmatrix}
0&1&0\\
1&0&0\\
0&0&0
\end{pmatrix}
\
\text{is covered by}
\
x=\begin{pmatrix}
0&0&1\\
0&0&0\\
1&0&0
\end{pmatrix},
}$
\item $\ds{
y=\begin{pmatrix}
0&0&1\\
0&0&0\\
1&0&0
\end{pmatrix}
\
\text{is covered by}
\
x=\begin{pmatrix}
0&0&0\\
0&0&1\\
0&1&0
\end{pmatrix}.}
$
\end{enumerate}
\end{Example}
When a down/right move is performed on $y$ (as in part 2. of Example~\ref{E:3}),
there may not be any available positions to place the non-zero entries of $x$. 
In this case, the pushed entries are placed on the diagonal. If there
are no available diagonal entries for both of the 1's, then one of them 
is pushed out of the matrix.  
\begin{Example}\label{E:4}
Once again, there are two moves of similar nature:
\begin{enumerate}
\item 
$\ds{
y=\begin{pmatrix}
0& 0&1&0\\
0& 1&0&0\\
1& 0&0&0 \\
0&0&0&0
\end{pmatrix}
\
\text{is covered by}
\
x=\begin{pmatrix}
0&0&0 &0\\
0&1&0 &0\\
0&0&1& 0\\
0& 0 & 0& 1
\end{pmatrix}.}
$
\item 
$\ds{
y=\begin{pmatrix}
0&1\\
1&0
\end{pmatrix}
\
\text{is covered by}
\
x=\begin{pmatrix}
0&0\\
0&1
\end{pmatrix}.}$
\end{enumerate}
\end{Example}

In the light of the above examples, we label a covering relation $x\rightarrow y$
in $PI_n$ as follows.
\begin{Definition}\label{D:coverings}
\begin{enumerate}
\item As in Example~\ref{E:1}, if the covering relation $x\rightarrow y$ 
is derived from the covering relation $\tilde{x} \rightarrow \tilde{y}$ of 
invertible involutions that are obtained from $x$ and $y$, respectively, 
then we use the labeling $\tilde{x} \rightarrow \tilde{y}$ as defined 
in~\cite{Incitti04}.
\item If the covering relation results from a move as in Example~\ref{E:2}, 
namely from a diagonal push where the element that is pushed from is at 
the position $(i,i)$, then we label it by $(i,i)$. 
\item Suppose $x\rightarrow y$ is as in Example~\ref{E:3}, or~\ref{E:4}.
Observe that, in all of these covering relations, one of the 1's is pushed 
down and the other is pushed right. Let $i$ denote the column index of 
the first 1 that is pushed to the right, and let $j$ denote the index of the 
resulting column. Then we label the covering by $(i,j)$. 
\end{enumerate}
\end{Definition}

To illustrate the third labeling let us present a few more examples. 

\begin{Example}
$$
y=\begin{pmatrix}
0&0&0&1&0\\
0&0&1&0&0\\
0&1&0&0&0\\
1&0&0&0&0\\
0&0&0&0&0
\end{pmatrix}
\
\text{is covered by}
\
x=\begin{pmatrix}
0&0&0&1&0\\
0&0&0&0&1\\
0&0&0&0&0\\
1&0&0&0&0\\
0&1&0&0&0\end{pmatrix}
$$
The corresponding labeling here is $(3,5)$.
\end{Example}
\begin{Example}
$$
y=\begin{pmatrix}
0&0&0&0&0&1\\
0&0&0&0&1&0\\
0&0&0&0&0&0\\
0&0&0&0&0&0\\
0&1&0&0&0&0\\
1&0&0&0&0&0
\end{pmatrix}
\
\text{is covered by}
\
x=\begin{pmatrix}
0&0&0&0&0&0\\
0&0&0&0&1&0\\
0&0&0&0&0&1\\
0&0&0&0&0&0\\
0&1&0&0&0&0\\
0&0&1&0&0&0
\end{pmatrix}
$$
The corresponding labeling here is $(1,3)$.

\end{Example}
\begin{Example}\label{E:y to z}
$$
y=\begin{pmatrix}
0&0&0&1&0\\
0&0&1&0&0\\
0&1&0&0&0\\
1&0&0&0&0\\
0&0&0&0&0
\end{pmatrix}
\
\text{is covered by}
\
x=\begin{pmatrix}
0&0&0&1&0\\
0&0&0&0&0\\
0&0&1&0&0\\
1&0&0&0&0\\
0&0&0&0&1
\end{pmatrix}
$$
The corresponding labeling here is $(2,3)$.

\end{Example}

\begin{Definition}\label{D:cdr}
If $x$ covers $y$ with label $(i,j)$, then we refer to it as an $(i,j)$-covering
and say that $y$ is obtained from $x$ by an $(i,j)$-move.
More briefly, we call a covering relation a {\em $c$-cover}, 
if it is derived from an involution; 
a {\em $d$-cover}, if it is obtained by a shift of a diagonal element; 
an {\em $r$-cover}, if it is derived from a right/down, or from 
a down/right move. The corresponding moves of 1's are 
referred to as $c$-, $d$- and $r$-moves.
\end{Definition}


\begin{Lemma}[Lemma 16,~\cite{CanTwelbeck}]\label{L:explicit covering relations for PI_n} 
Let $x$ and $y$ be two partial involutions. Then $x$ covers $y$ if and  
only if one of the following is true:

\begin{enumerate}
\item $x$ is obtained from $y$ by a $c$-move as in Example~\ref{E:1}.
\item Without removing a suitable rise, $x$ is obtained from 
$y$ by one of the following moves:
\begin{enumerate}
\item a $d$-move, as in Example~\ref{E:2}, 
\item an $r$-move, as in Example~\ref{E:3}, or as in Example~\ref{E:4}.
\end {enumerate}
\end{enumerate}
\end{Lemma}

It is shown in~\cite{CanTwelbeck} that the covering labelings defined in 
Definition~\ref{D:cdr} is an EL-labeling for $PI_n$.

\section{An EL-labeling of $PF_n$}
\label{S:main}

Covering relations of $F_n$ are covering relations in $I_n$, as well. 
Unfortunately, this is not the case for $PF_n$ relative to $PI_n$.
In other words, as a subposet of $PI_n$, $PF_n$ is not connected.  
For example, when $n=2$, there are only two partial fixed-point-free involutions:
$
x=\begin{pmatrix}
0&0\\
0&0
\end{pmatrix}
\
\text{and}
\
y=\begin{pmatrix}
0&1\\
1&0
\end{pmatrix}
$, hence $x$ covers $y$ as a partial fixed-point-free involution.
However, viewed as a partial involution $x$ does not cover $y$ since 
$y< \begin{pmatrix} 0&0\\
0&1\end{pmatrix} <x$.

\begin{Lemma}\label{L:coverings of PF_n}
Suppose $x\rightarrow y$ in $PF_n$. Then 
either $x$ covers $y$ as an element of $PI_n$, or there exists $z\in PI_n$ such that
$x \rightarrow z$ by an $d$-cover as an element of $PI_n$, 
and $z \rightarrow y$ by an $r$-cover in $PI_n$,  
where at each step the rank drops by 1.
Furthermore, in the first case, there are two possibilities:
\begin{enumerate}
\item
$x \rightarrow y$ is an $r$-cover in $PI_n$, or 
\item
$x \rightarrow y$ is a $c$-cover corresponding to a non-crossing $ee$, 
or to an $ed$-rise in $PI_n$.
\end{enumerate}
\end{Lemma}

\begin{proof}
Obviously, if $x$ covers $y$ in $PI_n$ and if both $x$ and $y$ 
are members of $PF_n$, then $x$ covers $y$ in $PF_n$, also. 
Thus, the last assertion follows from Lemma~\ref{L:explicit covering relations for PI_n}

We proceed with the assumption that $x,y\in PF_n$ but $x$ does 
not cover $y$ in $PI_n$. Towards a contradiction, assume that there 
does not exists $z\in PI_n$ as in the conclusion of the lemma. 
This means that the open interval $(y,x)=\{ z\in PI_n:\ y< z < x\}$ lies in 
$PI_n \setminus PF_n$.
In other words, any $z\in (y,x)$ has to have a non-zero diagonal entry. 
This eliminates the possibility of $z\rightarrow y$ being a $c$-cover 
(see Figure~\ref{CTofIncitti}). Clearly, $z\rightarrow y$ cannot be 
a $d$-cover, neither.

We continue with the assumption that $z$ is obtained from $y$ by an $r$-move,  
which places two symmetric entries on the diagonal. 
In this case, another $r$-move is possible in $y$ involving the same 1's. 
(To construct an example to this situation, start with $y$ as in Example~\ref{E:y to z}.)
Let $z_1$ denote this new element from $PF_n$. 
By comparing their rank-control matrices, we see that $Rk(x) < Rk(z_1)$, 
hence $y < z_1 < x$. This contradicts with our assumption that the interval 
$(y,x)$ lies in $PI_n \setminus PF_n$. Therefore, $z$ covers $y$ by an $r$-move, 
by deleting a 1 from $y$ and placing another to diagonal. 
Then by a $d$-move removing this diagonal 1 we obtain $x$. 
Thus we obtain a contradiction to our initial assumption. 

\end{proof}

\begin{Remark}\label{R:one line notation}
Let $x$ and $y$ be two elements from $PF_n$ such that $x$ covers 
$y$ by an $r$-move. Let $x=(x_1,\dots,x_n)$ and $y=(y_1,\dots, y_n)$ 
denote $x$ and $y$ in one-line notation. Then exactly one of the following 
statements is true:
\begin{enumerate}
\item $x$ is obtained from $y$ by replacing exactly two entries of $y=(y_1,\dots, y_n)$ by 0's.
\item There exists $i\in [n]$ such that $x$ is obtained from $y$ by replacing $y_i$ by the number 
$x_i$, setting $y_i$-th entry of $y$ to 0 and replacing the $x_i$-th entry of $y$ (which is a 0) by $i$.
\end{enumerate}

\end{Remark}


In the light of Lemma~\ref{L:coverings of PF_n} we make the following definition.
\begin{Definition}\label{ELPFPFI}
\begin{enumerate}
\item If the covering relation is derived from a $c$-move, then we use the 
labeling as defined in \cite{CanTwelbeck} and transform this 
label $(i,j)$ into $(n-i,n-j)$.
\item If the covering relation $x\rightarrow y$ results from an $r$-move, 
then we define the label to be $(i+n,j)$, where $x>y$ results from $y$ 
by moving the 1 in column $i$ to row $j$. 
If the 1 is pushed out of the matrix, then we set $j=n+1$.
\end{enumerate}
\end{Definition}
In the case of invertible fixed-point-free involutions we show in 
\cite{CanCherniavskyTwelbeck} that the lexicographically largest chain 
is the only decreasing chain.
Since the label is transformed from $(i,j)$ to $(n-i,n-j)$
 now the lexicographically smallest chain is increasing.
The reason the label of $r$-moves is shifted by $n$ in the first coordinate 
is to ensure that every $r$-cover has a bigger label than any $c$-cover.
In Figure~\ref{F:Labeling of PF4}, we illustrate the Definition~\ref{ELPFPFI}.
\begin{figure}[htp]
\begin{center}
\begin{tikzpicture}[scale=.3]
\node at (0,0) (a) {$(2,1,4,3)$};
\node at (-6,5) (b1) {$(3,4,1,2)$};
\node at (6,5) (b2) {$(2,1,0,0)$};
\node at (-6,10) (c1) {$(4,3,2,1)$};
\node at (6,10) (c2) {$(3,0,1,0)$};
\node at (-6,15) (d1) {$(4,0,0,1)$};
\node at (6,15) (d2) {$(0,3,2,0)$};
\node at (0,20) (e) {$(0,4,0,2)$};
\node at (0,25) (f) {$(0,0,4,3)$};
\node at (0,30) (g) {$(0,0,0,0)$};
\node[red] at (-3,2.5) {\scriptsize{(3,0)}};
\node[red] at (3,2.5) {\scriptsize{(8,5)}};
\node[red] at (-6,7.5) {\scriptsize{(3,2)}};
\node[red] at (0,7.5) {\scriptsize{(8,5)}};
\node[red] at (6,7.5) {\scriptsize{(5,3)}};
\node[red] at (-6,12.5) {\scriptsize{(7,5)}};
\node[red] at (6,12.5) {\scriptsize{(7,2)}};
\node[red] at (-2,14) {\scriptsize{(5,4)}};
\node[red] at (2,14) {\scriptsize{(8,5)}};
\node[red] at (-3,17.5) {\scriptsize{(8,2)}};
\node[red] at (3,17.5) {\scriptsize{(6,4)}};
\node[red] at (0,22.5) {\scriptsize{(8,3)}};
\node[red] at (0,27.5) {\scriptsize{(8,5)}};
\draw[-, very thick,color=blue](a) -- (b1);
\draw[-,thick](a) -- (b2);
\draw[-, very thick,color=blue](b1) -- (c1);
\draw[-,thick](b1) -- (c2);
\draw[-,thick](b2) -- (c2);
\draw[-, very thick,color=blue] (c1) -- (d1);
\draw[-,thick](c1) -- (d2);
\draw[-,thick](c2) -- (d1);
\draw[-,thick](c2) -- (d2);
\draw[-, very thick,color=blue] (d1) -- (e);
\draw[-,thick] (d2) -- (e);
\draw[-, very thick,color=blue] (e) -- (f);
\draw[-, very thick,color=blue] (f) -- (g);
\end{tikzpicture}
\end{center}
\caption{The EL-labeling of $PF_4$.}
\label{F:Labeling of PF4}
\end{figure}
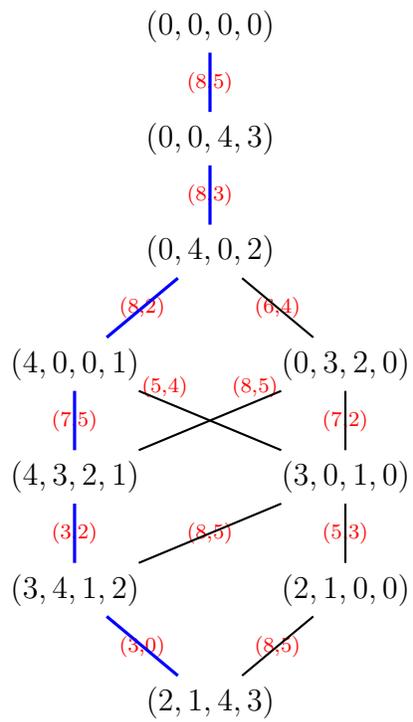


\begin{Proposition}\label{P:first part}
Let $y< x$ be two partial fixed-point-free involutions from $PF_n$, 
and let 
$$
\mf{c}: x = x_1 < x_2 < \cdots < x_{s+1} = y
$$
denote the maximal chain whose sequence of labels $f(\mf{c})$, 
as defined in Definition~\ref{ELPFPFI}, is lexicographically smallest among 
all such sequences. Then $f(\mf{c})$ is a weakly increasing sequence. 
\end{Proposition}

\begin{proof} 

Towards a contradiction assume that $f(\mf{c})$ is not weakly increasing. 
Then there exist three consecutive terms
\begin{equation*}
x_{t-1} <  x_t <  x_{t+1}
\end{equation*}
in $\mathfrak{c}$ such that $f((x_{t-1},x_t)) > f((x_t,x_{t+1}))$.
We have 4 cases to consider:
\begin{center}
Case 1: $type((x_{t-1},x_t))=c$, and $type((x_t,x_{t+1}))=c$,\\
Case 2: $type((x_{t-1},x_t))=r$, and $type((x_t,x_{t+1}))=r$,\\
Case 3: $type((x_{t-1},x_t))=c$, and $type((x_t,x_{t+1}))=r$,\\
Case 4: $type((x_{t-1},x_t))=r$, and $type((x_t,x_{t+1}))=c$.\\
\end{center}

In each of these cases, we either produce an immediate contradiction 
by showing that the two moves are interchangeable
(hence $\mf{c}$ is not the smallest chain), or we construct an 
element $z\in [x,y] \cap PF_n$ which covers $x_{t-1}$, and such that
$f((x_{t-1},z)) < f((x_{t-1},x_t))$. Since we assume that
$f(\mf{c})$ is the lexicographically smallest Jordan-H\"{o}lder sequence, 
the existence of such an element $z$ is a contradiction, also.

To this end, suppose that the label of the first move ($x_t \rightarrow x_{t-1}$) 
is $(i,j)$, and the second move ($x_{t+1} \rightarrow x_{t}$) is labeled by $(k,l)$.

{\em Case 1:} Follows from the proof for invertible fixed-point-free involutions.

{\em Case 2:} If $i=k$, then $l>j$. In this case, we interchange the two moves 
to obtain our desired contradiction. Therefore we continue with assuming $k<i$.
If $k-n=j$ then $j<i-n$ and $(m+n,l)$ is possible in $x_{t-1}$ with $m<j<i$, where 
$(m,i-n)$ is the position of the 1 in $x_{t-1}$.
If $k-n \neq j$ then either the two moves are interchangeable, or $(k,l)$ removes 
a suitable rise in $x_{t-1}$ which corresponds to a move with a smaller label than $(i,j)$.

{\em Case 3:} This case is impossible since every $c$-move has a smaller 
label than any $r$-move.

{\em Case 4:} If the $r$-cover labeled $(i,j)$ is the covering relation with the 
lexicographically smallest label then there is no suitable rise in $x_{t-1}$.
The $c$-move has to involve one of the moved 1's since otherwise there is 
a suitable rise in $x_{t-1}$. For this, one of the moved 1's has to have a 1 to 
the upper left or the lower right in $x_t$ that was not to the upper left or lower 
right of it in $x_{t-1}$. Since the 1's are moved right and down respectively, 
it is impossible that there is a 1 to the lower right in $x_t$ that is not to the lower right in $x_{t-1}$.
If the $c$-cover corresponds to the suitable rise $(m,i-n)$ (with label $(n-m,i)$), 
then $(i,j)$ is not the $r$-move with the smallest label in $x_{t-1}$ since in this case 
$(m+n,j)$ is possible in $x_{t-1}$ with $(n+m,j)<(i,j)$.
If the $c$-cover corresponds to the rise $(m,j)$, then the $r$-move $(m+n,i-n)$ 
is possible in $x_{t-1}$ which again has a smaller label than $(i,j)$.

\end{proof}

\begin{Proposition}
We retain the notation from (the proof of) Proposition~\ref{P:first part}.
Then $f(\mf{c})$ is the unique increasing chain in $[y,x]$. 
\end{Proposition}

\begin{proof}

We use induction on the length $s+1$ of the interval $[y,x]$ 
to prove that no other chain is lexicographically increasing.
Clearly, if $x$ covers $y$, there is nothing to prove, so, 
we assume that for any interval of length $k\leq s$ 
there exists a unique increasing maximal chain.

Assume that there exists another increasing chain
\begin{equation*}
\mathfrak{c}': y=x_0< x_1' < \cdots <x_s' <  x_{s+1}=x.
\end{equation*}
Since the length of the chain
\begin{equation*}
x_1'< \cdots < x_s' < x_{s+1}=x
\end{equation*}
is $s$, by the induction hypothesis, it is the lexicographically smallest chain 
between $x_1'$ and $x$.
We are going to find contradictions to each of the following possibilities:
\begin{center}
Case 1: $type(x_0,x_1)=c$, and $type(x_0,x_1')=c$,\\
Case 2: $type(x_0,x_1)=r$, and $type(x_0,x_1')=r$, \\
Case 3: $type(x_0,x_1)=c$, and $type(x_0,x_1')=r$,\\
Case 4: $type(x_0,x_1)=r$, and $type(x_0,x_1')=c$.\\
\end{center}

In each of these cases we will construct a partial fixed-point-free involution $z \in [y,x]$ such
that $z$ covers $x_1'$ and $f((x_1',z))<f((x_1',x_2'))$, contradicting the induction hypothesis.
To this end, let $f((x_0,x_1))=(i,j)$, $f((x_0,x_1'))=(k,l)$ and assume that $(k,l)<(i,j)$.

{\em Case 1:} Done in the proof for the invertible case.

{\em Case 2:} It is impossible for $i=k$ since there is only one $r$-move for each 1.
Therefore assume that $i<k$. Let the moved 1's be on the symmetric positions 
$(i-n,m)$ and $(m,i-n)$ in $x_0$. If $k=m+n$ then $(l+n,j)$ is possible in $x_1'$ 
with $(l+n,j)<(k,l)$. If $ k \neq m$ then either the two moves are interchangeable 
or the suitable rise $(n-i,n-k)$ is possible in $x_1'$.

{\em Case 3:} Since no $r$-move can remove a suitable rise, 
there exists a legal $c$-move in $x_1'$. But this $c$-move has a 
smaller label than $(k,l)$ which is our desired contradiction.

{\em Case 4:} This case is not possible because every $c$-move has 
a smaller label than any $r$-move.

\end{proof}

Combining previous two propositions, we have our first main result:
\begin{Theorem}\label{T:first main}
The poset $PF_n$ is an EL-shellable poset.
\end{Theorem}

\section{The order complex of $PF_{n}$}
\label{S:order complex}

In~\cite{CanCherniavskyTwelbeck}, it is shown that the order complex $\Delta(F_n)$ 
of fixed-point-free involutions triangulates a ball of dimension $n^2-n-2$. 
In this section we obtain a similar result for $PF_n$. 

\begin{Lemma}\label{L:dimension formula} For all $n\geq 2$, 
$$
\dim \Delta(PF_n)=\ell (PF_n)=n+(n-1)+\dots+1-n={n\choose2}.
$$
\end{Lemma}
\begin{proof}
Straightforward by using~(\ref{A:a formula for ell}).
\end{proof}

We continue by analyzing the intervals of length two.
\begin{Lemma}\label{L:at most two}
Each length two interval $[y,x] \subseteq PF_n$ has at most four, at least three elements. 
\end{Lemma} 
\begin{proof}
Just as in the proof of Theorem~\ref{T:first main}, if $y< z < x$, then 
there are 4 cases to consider:
\begin{center}
Case 1: $type((y,z))=c$, and $type((z,x))=c$,\\
Case 2: $type((y,z))=r$, and $type((z,x))=r$,\\
Case 3: $type((y,z))=c$, and $type((z,x))=r$,\\
Case 4: $type((y,z))=r$, and $type((z,x))=c$.\\
\end{center}
In the first case, $[y,x]$ is isomorphic to an interval in $F_m$ for some $m\leq n$, 
and therefore, it has at most 4 elements (since $F_m$ is a connected subposet of $I_m$,
which is Eulerian).

In the second case, we look at the one-line notations of $y$ and $x$. 
See~\ref{R:one-line notation}. 
If $z$ is obtained from $y$ by setting two non-zero entries of $y$ to 0's, and if,
at the same time, $x$ is obtained from $z$ by setting two non-zero entries of $z$ to 0's,
then $y$ and $x$ differ at exactly 4 entries. Therefore, $[y,x]$ contains at most one other 
element other than $z$, which is obtained from $y$ by setting two entries of $y$ to 0's. 
If $z$ is obtained by increasing the $i$-th entry $y_i$ of $y$ to $z_i$, 
and if, at the same time, $x$ is obtained from $z$ by increasing the $i$-th
entry $z_i$ of $z$ to $x_i$, then $[y,x]$ has exactly 3 elements. 
If $z$ is obtained from $y$ by increasing the $i$-th entry $y_i$ of $y$ to $z_i$, 
and if $x$ is obtained from $z$ with no overlap with the replaced/increased entries of $y$,
then $[y,x]$ has exactly 4 elements. 
Finally, if $z$ is obtained by increasing the $i$-th entry $y_i$ of $y$ to $z_i$, 
and $x$ is obtained from $z$ by replacing the $z_i$-th entry of $z$ by 0,
then $y$ and $x$ differ at exactly at 4 positions. Therefore, the interval $[y,x]$
have at most 4 elements.

Since the arguments of Case 3 and Case 4 are identical, we handle Case 3 only. 
Suppose that there exist more than 4 elements in $[y,x]$. Since one of the elements
$y < z < x$ is obtained from $y$ by a $c$-move, the covering type of any other 
$y< z_1 < x$ is not of type $c$. Otherwise, to obtain $x$ from $z$ we need to 
apply another $c$-move to $z$. But then the matrix ranks of $y$ and $x$ would be the 
same. Therefore, we conclude that if $z_1 \neq z$ and $y< z_1 < x$, then $z_1$
is obtained from $y$ by an $r$-move, and $x$ is obtained from $z_1$ by a $c$-move. 
Now it is clear that it is impossible to have another element $y< z_2 < x$ such that 
$z_2$ covers $y$ by an $r$-move and $z_2 \notin \{ z, z_1 \}$. 
Therefore $[y,x]$ have exactly 4 elements and the proof is complete. 

\end{proof}

We know from~\cite{DK74} that a pure, shellable simplicial complex $\Delta$ of which 
every $\dim \Delta -1$ face is contained in at most two facets is homeomorphic 
to either a ball, or a sphere. 
By Lemma~\ref{L:at most two}, we see that $\Delta(PF_n)$ satisfies this property.

\begin{Theorem}\label{T:second main}
Let $\widetilde{PF}_n$ denote the proper part of $PF_n$, 
namely the subposet obtained from $PF_n$ by removing its smallest
and the largest elements. 
For $n\geq 3$, the order complex $\Delta(\widetilde{PF}_n)$ triangulates a ball of dimension 
$\dim \Delta(PF_n)-2={n\choose2}-2$.
\end{Theorem}

\begin{proof}

By the discussion above, it is enough to show that the reduced Euler characteristic 
of $\Delta(\widetilde{PF}_n)$ is 0.

By Hall's Theorem (see Chapter 3,~\cite{StanleyEC1}), we know that the reduced 
Euler characteristic of an order complex of a poset $P$ is equal to the 
value of the M\"obius function $\mu_{\widehat{P}}$ on the interval $[\hat{0},\hat{1}]$, 
where $\widehat{P}$ is $P$ with a $\hat{0}$ (a smallest element) and a $\hat{1}$ 
(a largest element) adjoined. Therefore, it is enough to show that 
$\mu_{{PF_n}}([\hat{0},\hat{1}])=0$, where $\hat{0}=(0,\dots,0)$ and $\hat{1} =(0,\dots, 0,n,n-1)$.

Let $PF_n^*$ denote the dual of $PF_n$. By abuse of notation we use $\hat{0}$ 
for the smallest element of $PF^*_n$ although it is $\hat{1}$ of $PF_n$. 
Similarly, we denote the largest element of $PF_n^*$ by $\hat{1}$. 
Now, since $\mu_{{PF_n}}([\hat{0},\hat{1}])=\mu_{{PF_n^*}}([\hat{0},\hat{1}])$, we 
are going to show that the later value is 0.

It is easy to see that the cardinality of the set 
$\{ x \in {PF_n^*}:\ \ell_{{PF_n^*}}([\hat{0},x]) \leq 3 \}$ is 1, 
for $n\geq 3$. Indeed, if $\ell_{{PF_n^*}}(x) = 3$, then in one-line notation 
$x = (0,\dots, 0,n,0,n-2)$, and $[\hat{0},x]=\{ \hat{0} < 0 < z_0 < x\}$, where 
$z_0 = (0,\dots, 0, n,n-1)$.

For simplicity, let us denote $\mu_{{PF_n^*}}$ by $\mu$, and denote
the length function $\ell_{{PF_n^*}}$ by $\ell$. 
We prove by induction that $\mu([\hat{0},z])=0$ for all $z$ with $\ell(z)>1$.
Our base case is when $\ell(z) = 2$. In this case, $[\hat{0},z]$ is a chain of length 2 
by the discussion in the previous paragraph, and hence, the corresponding value is 0. 
Now assume that $\mu ([\hat{0}, z]) = 0$ for 
all $z$ with $2\leq \ell(z) \leq s$, and let $z'\in {PF_n^*}$ be an element with 
$\ell(z') = s+1$. 
Since 
$$
\mu([\hat{0},z'])=-\sum_{\hat{0} \leq z<z'} \mu([\hat{0},z])=
- (\mu([\hat{0},\hat{0}])+\mu([\hat{0},0]) = -(1 + (-1)) = 0,
$$
the proof is complete. 

\end{proof}

\section{Length-generating functions}
\label{S:length generating function}

Recall that the {\em standard form} of an involution $\pi\in I_n$ is a 
product of transpositions of the form 
\begin{align}\label{A:product of transposition}
\pi=\left(i_1,j_1\right)\left(i_2,j_2\right)\cdots\left(i_m,j_m\right),
\end{align}
where for all $1\leqslant t\leqslant m$, $i_t<j_t$ and $i_1<i_2<\cdots<i_m$.
We call the transpositions appearing in (\ref{A:product of transposition}) 
as {\em arcs}.
Using bijection~(\ref{bijection}) from the Introduction section, we identify the elements 
of $PF_n$ as involutions in $S_n$. With this identification, let us denote by $I(n,k)$ 
the set of involutions of $S_n$ having $k$ arcs, and define its length generating function by 
$$
\mf i_q(n,k) :=\sum_{\pi\in I(n,k)}q^{\ell_{PF_n}(\pi)}. 
$$

Recall also that the $q$-analog of a natural number $n\in \N$ is the polynomial 
$[n]_q = 1+q+\cdots +q^{n-1}$.

\begin{Proposition}\label{recfl} 
For all $n\geq 2$ and $k\in \{2,\dots,n\}$, we have 
$$
\mathfrak i_q(n+1,k)=q^n\mathfrak i_q(n,k)+[n]_q\mathfrak i_q(n-1,k-1)\,.
$$
\end{Proposition}

\begin{proof}
We begin with defining a bijection:
$$
\Phi\,:\,I(n+1,k)\rightarrow I(n,k)\bigcup \left(\{2,3,\dots,n,n+1\}\times I(n-1,k-1)\right)\,.
$$
Let $\pi$ be an element of $I(n+1,k)$. If $\pi(1)=1$, then we define 
$\Phi(\pi)=\sigma\in I(n,k)$ as follows: $\sigma(j)=\pi(j+1)$ for $j\in\{1,2,\dots,n\}$. 
In other words, in matrix notation, $\sigma$ is obtained from $\pi$
by deleting its first row and its first column. 
Notice that if $\pi(1)=1$, then $\ell_{PF_{n+1}}(\pi)=\ell_{PF_n}(\sigma)+n$,  
since when we delete the first zero row from the rank-control matrix, the 
parameter $\rho_<$ decreases by $n$, which is the number of zeros in this 
row in positions from 2 to $n+1$.

Suppose now that $\pi(1)=i\in\{2,3,\dots,n+1\}$. In this case,  
we define $\Phi(\pi)$ to be the pair $\Phi(\pi)=(i,\sigma)$, where 
$\sigma$ is the involution from $I(n-1,k-1)$ defined by
$$
\sigma(j)=
\begin{cases}
\pi(j+1) & \text{ if }  j\in\{1,..,i-2\}, \\
\pi(j+2) & \text{ if }  j\in\{i-1,..,n-1\}.
\end{cases}
$$
In matrix notation, $\sigma$ is obtained from $\pi$ by deleting the first 
and the $i$-th rows of $\pi$, as well as deleting its first and $i$-th columns.
In this case we have: $\ell_{PF_{n+1}}(\pi)=\ell_{PF_{n-1}}(\sigma)+i-2$. 
To see this, notice that all the equalities in the upper triangular 
portion of $Rk(\pi)$ are carried into that of $Rk(\sigma )$ with 
additional $i-2$ equalities arising from the 0's at the positions $(1,2)$, $(1,3)$,\dots,$(1,i-1)$ 
of $\pi$. Thus $\rho_< (\pi) = \rho_<(\sigma) + i-1$. On the other hand, since 
the ranks of $\pi$ and $\sigma$ differ by 2, and their sizes differ by 2, 
by the formula (\ref{A:a formula for ell}), we see that 
\begin{align}
\ell_{PF_{n+1}}(\pi) &=\rho_<(\pi)-\frac{2(n+1)-rk(\pi)}{2} \notag \\
&=\rho_<(\sigma)+i-1-\frac{2(n-1)-rk(\sigma)+2}{2} \notag \\
&=\ell_{PF_{n-1}}(\sigma)+i-2. \label{A:an example}
\end{align}
See Example~\ref{E:an example} for an illustration.

Now, in the light of these observations, we derive the desired recurrence: 
\begin{align*}
\mathfrak i_q(n+1,k) &=\sum_{\pi\in I(n+1,k)}q^{\ell_{PF_{n+1}}(\pi)} \\
&=\sum_{\pi\in I(n+1,k),\pi(1)=1}q^{\ell_{PF_{n+1}}(\pi)}+
\sum_{\pi\in I(n+1,k),\pi(1)\neq 1}q^{\ell_{PF_{n+1}}(\pi)}\\
&=\sum_{\sigma\in I(n,k)}q^{\ell_{PF_n}(\sigma)+n}
+\sum_{i=2}^{n+1}\sum_{\sigma\in I(n-1,k-1)}q^{\ell_{PF_{n-1}}(\sigma)+i-2}\\
&=q^n\cdot\sum_{\sigma\in I(n,k)}q^{\ell_{PF_n}(\sigma)}
+\sum_{i=2}^{n+1}q^{i-2}\cdot\sum_{\sigma\in I(n-1,k-1)}q^{\ell_{PF_{n-1}}(\sigma)}\\
&=q^n\mathfrak i_q(n,k)+(1+q+q^2+\cdots +q^{n-1})\mathfrak i_q(n-1,k-1)\\
&=q^n\mathfrak i_q(n,k)+[n]_q\mathfrak i_q(n-1,k-1).
\end{align*}
\end{proof}

\begin{Example}\label{E:an example}
Let us consider an example in order to understand (\ref{A:an example}). Consider 
 $$
 \pi=\begin{pmatrix}
 0 &0 &0 &0 &1 &0\\
 0 &0 &0 &0 &0 &0\\
 0 &0 &0 &0 &0 &0\\
 0 &0 &0 &0 &0 &1\\
 1 &0 &0 &0 &0 &0\\
 0 &0 &0 &1 &0 &0
 \end{pmatrix}\,\,\,\textrm{with}\,\,\,
 Rk(\pi)=\begin{pmatrix}
 0 &0 &0 &0 &1 &1\\
 0 &0 &0 &0 &1 &1\\
 0 &0 &0 &0 &1 &1\\
 0 &0 &0 &0 &1 &2\\
 1 &1 &1 &1 &2 &3\\
 1 &1 &1 &2 &3 &4
 \end{pmatrix}.
 $$
 $$
 \sigma=\Phi(\pi)=\begin{pmatrix}
 0 &0 &0 &0\\
 0 &0 &0 &0\\
 0 &0 &0 &1\\
 0 &0 &1 &0\end{pmatrix}\,\,\,\textrm{with}\,\,\,
 Rk(\sigma)=\begin{pmatrix}
 0 &0 &0 &0\\
 0 &0 &0 &0\\
 0 &0 &0 &1\\
 0 &0 &1 &2
 \end{pmatrix}.
 $$
By using (\ref{A:a formula for ell}), it is easy to verify 
$\ell_{PF_6}(\pi) = 12- 4=8$ and $\ell_{PF_4}(\sigma)= 8 - 3=5$.
 
\end{Example}

\subsection{An explicit formula for $\mathfrak i_q(n,k)$.}

Let $\pi \in PF_n$ be a partial involution and let
$\pi=\left(i_1,j_1\right)\left(i_2,j_2\right)\cdots\left(i_m,j_m\right)$ denote its standard form
viewed as an involution in $I_n$ via bijection (\ref{bijection}).
It follows from the proof of Proposition~6.2 of~\cite{Cherniavsky11} that 
the following equality is true:
\begin{equation}\label{LengthFormula}
\ell_{PF_n}(\pi)=\rho_<(\pi)=\widetilde{inv}(\pi)+\sum_{a\,:\,\pi(a)=a}(n-a)\,,
\end{equation}
where $\widetilde{inv}(\pi)$ is the ``modified inversion number,'' which is equal to 
the number of inversions in the word $i_1j_1i_2j_2\cdots i_mj_m$.

\begin{Proposition}
$\mathfrak i_q(2k,k)=[2k-1]_q!!$.
\end{Proposition}
\begin{proof} By Proposition~\ref{recfl} we have
$$
\mathfrak i_q(2k,k)=q^{2n-1}\mathfrak i_q(2k-1,k)+[2k-1]_q\mathfrak i_q(2k-2,k-1).
$$
Since there are no involutions in $S_{2k-1}$ which have $k$ arcs 
(the maximal number of arcs for an involution in $S_{2k-1}$ is $k-1$), 
we have $ \mathfrak i_q(2k-1,k)=0$ and therefore 
$\mathfrak i_q(2k,k)=[2k-1]_q\mathfrak i_q(2k-2,k-1)=[2k-1]_q\mathfrak i_q\left(2(k-1),k-1\right)$. 
Now, by induction we get $\mathfrak i_q(2k,k)=[2k-1]_q!!$. 
\end{proof}

\begin{Proposition}
$$
\mathfrak i_q(n,k)=q^{{{n-2k}\choose 2}}\cdot{n\choose{2k}}_q\cdot [2k-1]_q!!,
$$
where ${n\choose{2k}}_q=\frac{[n]_q!}{[2k]_q![n-2k]_q!}$.
\end{Proposition}

\begin{proof}
Let $\pi$ an element from $I(n,k)$. The involution $\pi\in S_n$ has $k$ arcs, 
hence, it has $n-2k$ fixed points. Thus, $n-2k$ zero rows and columns 
in the corresponding partial fixed-point-free involution matrix. 
So, there is a natural bijection
$$
\pi\leftrightarrow \left(\{i_1,\dots,i_{n-2k}\},\ \sigma\right),
$$
where $1\leqslant i_1<i_2<\cdots <i_{n-2k}\leqslant n$ are the fixed points 
of $\pi$ and $\sigma\in I(2k,k)$ is the fixed point free involution of $S_{2k}$, 
whose partial fixed-point-free involution matrix is obtained from $\pi$ 
by deleting zero rows and columns. Now, using formula~(\ref{LengthFormula}) 
we have
\begin{align}
\mathfrak i_q(n,k) &=\sum_{\pi\in I(n,k)}q^{\ell_{PF_n}(\pi)} \notag \\
&=\sum_{\stackrel{(\{i_1,\dots,i_{n-2k}\}\,\ \sigma):}
{1\leqslant i_1<i_2<\cdots <i_{n-2k}\leqslant n,\sigma\in I(2k,k)}}
q^{n-i_1+n-i_2+\cdots+n-i_{n-2k}+\ell_{F_{2k}}(\sigma)} \notag \\
&= \left( \sum_{1\leqslant i_1<i_2<\cdots <i_{n-2k}\leqslant n}
q^{n-i_1+\cdots+n-i_{n-2k}} \right)  \cdot 
\left(\sum_{\sigma\in I(2k,k)}q^{\ell_{F_{2k}}(\sigma)}\right) \notag \\
&= \left( \sum_{0\leqslant j_1<\cdots <j_{n-2k}\leqslant n-1}
q^{j_1+\cdots+j_{n-2k}} \right) \cdot \mathfrak i_q(2k,k). \label{sum over distinct}
\end{align}
To simplify (\ref{sum over distinct}), we use well known Gaussian identity 
(see \cite{StanleyEC1}, formula~(1.87)):
\begin{align}\label{Gaussian}
\prod_{i=0}^{j-1} (1+x q^i) = \sum_{k=0}^j x^k q^{k \choose 2} { j \choose k}_q,
\end{align}
which is equivalent, by expanding the product, to 
\begin{align}\label{formula follows}
\sum_{0\leqslant s_1<s_2<\cdots <s_{k}\leqslant j-1}q^{\sum_{r=1}^k s_r} x^k=
\sum_{k=0}^j x^k q^{k \choose 2} { j \choose k}_q.
\end{align}
Replacing $j$ by $n$ ,and comparing the coefficients of $x^{n-2k}$ in (\ref{formula follows}), 
we obtain our desired formula
\begin{align*}
\mathfrak i_q(n,k) =q^{{{n-2k}\choose 2}}\cdot{n\choose{2k}}_q\cdot [2k-1]_q!!.
\end{align*}
\end{proof}

\subsection{Length generating function of $PF_n$}

Next, we look at the length generating function of $PF_n$ more closely.
$$
\mathfrak p_q(n) :=\sum_{\pi\in PF_n}q^{\ell_{PF_n}(\pi)}
=\sum_{k=0}^{\lfloor \frac{n}{2}\rfloor} \mathfrak i_q(n,k)\,.
$$
By a straightforward calculation we see that 
$\mathfrak p_q(1)=1$, $\mathfrak p_q(2)=1+q$, $\mathfrak p_q(3)=1+q+q^2+q^3$.

\begin{Proposition}
For all $n\geq 2$, we have 
$$
\mathfrak p_q(n+1)=q^n\mathfrak p_q(n)+[n]_q\mathfrak p_q(n-1).
$$
\end{Proposition}

\begin{proof}
Follows from Proposition~\ref{recfl}.
\end{proof}

\begin{Example}
It is easy to verify the following calculation from the Hasse diagram of 
$PF_4$ in Figure~\ref{F:Labeling of PF4}:
$\mathfrak p_q(4)=q^3\mathfrak p_q(3)+[3]_q\mathfrak p_q(2)=1+2q+2q^2+2q^3+q^4+q^5+q^6$.
\end{Example}

\subsection{Skew-symmetric matrices over $\F_q$}

There is an interesting similarity between the rank generating function $\mathfrak i_q(n,k)$ 
and the number of $\F_q$-rational points of rank $2k$, $n\times n$ skew symmetric matrices,
which we denote by $\mt{Skew}_n^{2k}$. 
Here $\F_q$ is the finite field with $q$ elements. 
It is well known that the number of $\F_q$-rational points of the general 
linear group $\mt{GL}_n$ and the symplectic group 
$\mt{Sp}_{n}$ ($n=2m$) are given by 
$$
| \mt{GL}_n |_{\F_q}   = q^{{n\choose 2}}  \prod^n_{i=1} (q^{i}-1) \mt{ and }\
|\mt{Sp}_{2m} |_{\F_q}  = q^{m^2} \prod^m_{i=1} (q^{2i}-1).
$$
The group $G= \mt{GL}_n$ acts $\mt{Skew}_n^{2k}$ transitively. 
A simple matrix computation shows that 
$$
|G_x |_{\F_q} = | \mt{GL}_{n-2k} |_{\F_q}  |\mt{Sp}_{2k} |_{\F_q} |\mt{Mat}_{n-2k,2k} |_{\F_q},
$$
where $\mt{Mat}_{n-2k,2k}$ is the space of $2k \times (n-2k)$ matrices. 
Thus, 
\begin{align*}
|\mt{Skew}_n^{2k}|_{\F_q} = |G/G_x |_{\F_q} &= 
\frac{ q^{{n\choose 2}}  \prod^n_{i=1} (q^{i}-1) }{ q^{k^2} \prod^k_{i=1} (q^{2i}-1) 
q^{{n-2k\choose 2}}  \prod^{n-2k}_{i=1} (q^{i}-1) q^{2k(n-2k)}},
\end{align*}
which simplifies as follows
\begin{align*}
|\mt{Skew}_n^{2k}|_{\F_q} &= q^{{n\choose 2} - k^2- {n-2k\choose 2} -2k(n-2k)}  \frac{ [n]! (q-1)^n }{ 
(\prod^k_{i=1} [2i] ) (q-1)^k [n-2k]!  (q-1)^{n-2k} } \\
&= q^{2{k\choose 2} } \frac{ [n]! (q-1)^k}{ (\prod^k_{i=1} [2i] ) [n-2k]! } \\
&= q^{2{k\choose 2} } (q-1)^k {n \choose 2k}_q  [2k-1]!! \,.
\end{align*}
In other words, 
\begin{align*}\label{Perplexing}
|\mt{Skew}_n^{2k}|_{\F_q}  = \mathfrak i_q(n,k) q^{2{k \choose 2} -{{n-2k}\choose 2}} (q-1)^k \, .
\end{align*}

\bibliography{PFPF}
\bibliographystyle{plain}

\end{document}